\documentclass[11pt,twoside]{amsart}
\usepackage{color}
\usepackage{hyperref}
\usepackage[latin1]{inputenc}
\usepackage[OT1]{fontenc}
\usepackage{amsmath}
\usepackage{amsthm}
\usepackage{amssymb}
\usepackage[all]{xy}
\usepackage{xy}%HC
\usepackage{ifthen} %HC
\usepackage{hyperref} %HC
\usepackage{graphicx}
\usepackage{paralist}
\usepackage{stmaryrd}

\newtheorem{theorem}{Theorem}[section]

\newtheorem{proposition}[theorem]{Proposition}

\newtheorem{lemma}[theorem]{Lemma}

\numberwithin{equation}{section}

\theoremstyle{definition}
\newtheorem{definition}[theorem]{Definition}
\newtheorem{remark}[theorem]{Remark}

\newcommand{\CC}{\mathbb{C}}

\newcommand{\FF}{\mathbb{F}}

\newcommand{\PP}{\mathbb{P}}

\newcommand{\ZZ}{\mathbb{Z}}

\newcommand{\sO}{\mathcal{O}}
\newcommand{\sE}{\mathcal{E}}

\renewcommand{\to}{\xymatrix@1@=15pt{\ar[r]&}}
\renewcommand{\rightarrow}{\xymatrix@1@=15pt{\ar[r]&}}
\renewcommand{\mapsto}{\xymatrix@1@=15pt{\ar@{|->}[r]&}}
\renewcommand{\twoheadrightarrow}{\xymatrix@1@=15pt{\ar@{->>}[r]&}}
\renewcommand{\hookrightarrow}{\xymatrix@1@=15pt{\ar@{^(->}[r]&}}
\newcommand{\congpf}{\xymatrix@1@=15pt{\ar[r]^-\sim&}}

\begin{document}

\newboolean{xlabels} %HC
\newcommand{\xlabel}[1]{ %HC
                        \label{#1} %HC
                        \ifthenelse{\boolean{xlabels}} %HC
                                   {\marginpar[\hfill{\tiny #1}]{{\tiny #1}}} %HC
                                   {} %HC
                       } %HC
%\setboolean{xlabels}{true} %HC
\setboolean{xlabels}{false} %HC

\title[Stable rationality of conic bundles and Prym curves]{On stable rationality of some conic bundles and moduli spaces of Prym curves}

\author[B\"ohning]{Christian B\"ohning}
\address{Christian B\"ohning, Mathematics Institute, University of Warwick\\
Coventry CV4 7AL, England}
\email{C.Boehning@warwick.ac.uk}

\author[Bothmer]{Hans-Christian Graf von Bothmer}
\address{Hans-Christian Graf von Bothmer, Fachbereich Mathematik der Universit\"at Hamburg\\
Bundesstra\ss e 55\\
20146 Hamburg, Germany}
\email{hans.christian.v.bothmer@uni-hamburg.de}

%\thanks{$^1$ Supported by Heisenberg-Stipendium BO 3699/1-2 of the DFG (German Research Foundation) during the initial stages of this work}
%\thanks{$^2$ Partially supported by the RTG 1670 of the  DFG (German Research Foundation)}

\begin{abstract}
We prove that a very general hypersurface of bidegree $(2, n)$ in $\PP^2\times \PP^2$ for $n$ bigger than or equal to $2$ is not stably rational, using Voisin's method of integral Chow-theoretic decompositions of the diagonal and their preservation under mild degenerations. At the same time, we also analyse possible ways to degenerate Prym curves, and the way how various loci inside the moduli space of stable Prym curves are nested. No deformation theory of stacks or sheaves of Azumaya algebras like in recent work of Hasset-Kresch-Tschinkel is used, rather we employ a more elementary and explicit approach via Koszul complexes, which is enough to treat this special case.
\end{abstract}

\maketitle

\section{Introduction, description of the problem and prerequisites}\xlabel{sIntroduction}

In this article we work over the complex numbers $\CC$ throughout. A hypersurface $H_{2,n} \subset \PP^2_{(x:y:z)} \times \PP^2_{(u:v:w)}$ of bidegree $(2,n)$ is given by an equation
\[
(x,y,z) A(u,v,w) (x,y,z)^t =0
\]
where $A(u,v,w)$ is a symmetric $3\times 3$ matrix with entries homogeneous of degree $n$ in $u,v,w$, hence projection of a very general $H_{2,n}$ to $\PP^2_{u,v,w}$ realizes it as a conic bundle over $\PP^2_{(u,v,w)}$ with discriminant curve $\Delta_{H_{2,n}}= \det A(u,v,w)$ of degree $3n$. This discriminant curve carries a natural Prym structure, i.e. a two-torsion line bundle $\alpha$, arising from the determinantal representation of $\Delta_{H_{2,n}}$. More precisely, $\alpha$ has a minimal graded free resolution
\begin{gather}\label{fResAlpha}
\xymatrix{
0 \ar[r] & \mathcal{O}_{\PP^2}(-2n)^3 \ar[r]^{ A} & \mathcal{O}_{\PP^2}(-n)^3 \ar[r] & \alpha \ar[r] & 0 .
}
\end{gather}
It is more accurate to think of the pair $(\Delta_{H_{2,n}}, \alpha )$ of the discriminant curve or discriminant datum of the conic bundle. Note already at this point that if $H_{2,n}$ is not very general, $\alpha$ need not necessarily be a bundle at all, but might, for example, be just a symmetric torsion-free sheaf. Then our main result is

\begin{theorem}\xlabel{tMain}
For $n\ge 2$, the very general hypersurface $H_{2,n} \subset \PP^2_{(x:y:z)} \times \PP^2_{(u:v:w)}$ of bidegree $(2,n)$ is not stably rational.
\end{theorem}

We now describe the strategy of the proof and some subtle problems which arise in its implementation, the resolution of which can be said to be one of the main contributions of this article (the other one being the usage of a more explicit and low-tech type of deformation theory based on the Koszul complex, which avoids the arguments involving root stacks and sheaves of Azumaya algebras in \cite{HKT15} completely; we will say more about this below). These problems are related to the fact that the discriminant curves are not the generic plane Prym curves of degree $3n$ as soon as $n >2$.

\medskip

In \cite{Voi15} Voisin introduced a very powerful new degeneration technique that allows one to prove that very general members of certain families of ``nearly rational" (e.g. unirational) varieties are not stably rational; the idea is that stably rational varieties have an integral Chow-theoretic decomposition of the diagonal resp. universally trivial Chow group of zero cycles, and this property is preserved under mild degenerations. The technique was developed further, generalised substantially and cast in its natural theoretical framework in \cite{CT-P16}. This made possible a wealth of applications, some of them using degenerations in unequal characteristic such as \cite{To16}; without any pretense to completeness we just mention as examples \cite{CT-P15}, \cite{To16}, \cite{HKT15}, \cite{HT16}, \cite{HPT16a}, \cite{Pi16} and \cite{HPT16}. In this last article the authors exhibit a family of smooth varieties over a connected base some of whose fibers are rational whereas others are irrational. The existence of an integral Chow-theoretic decomposition of the diagonal is the only (stably) birational invariant so far that has been used to distinguish birational types of smoothly deformation-equivalent smooth varieties.

We have taken the formulation of the following result that encapsulates the method from \cite{Beau15}, but the result is really a simplified version of \cite[Thm.  1.14]{CT-P16}, and the reader is referred to the latter source for a proof.

\begin{theorem}\xlabel{tVoisin}
Let $B$ be a smooth variety and $o\in B$ a (closed) point. Suppose that $f\colon \mathcal{X} \to B$ is a flat projective morphism such that the generic fiber of $f$ is smooth and that the fiber $X:= \mathcal{X}_0$ is integral. Suppose $X$ admits a resolution of singularities $\sigma \colon \tilde{X} \to X$ with the following properties:
\begin{enumerate}
\item
The torsion subgroup of $H^3 (\tilde{X}, \ZZ )$ is nontrivial.
\item
The fiber of $\sigma$ over any scheme-theoretic point $\xi \in X$ is a smooth rational variety over the residue field $\kappa (\xi )$.
\end{enumerate}
Then for a very general point $b\in B$, the fiber $\mathcal{X}_b$ is not stably rational.
\end{theorem}

Here condition (1) can be replaced by any other condition that ensures that the Chow group of zero cycles is not universally trivial for $X=\mathcal{X}_0$; the condition means that the unramified Brauer group of $X$ is nonzero. 

To apply Theorem \ref{tVoisin} to show that certain very general members of families of conic bundles are not stably rational, one would like to construct a degeneration as conic bundles (so $X=\tilde{X}_0$ should retain a conic bundle structure), and then one needs a theorem that tells one when a conic bundle $X$ has a desingularization $\sigma \colon \tilde{X} \to X$ such that  (1) and (2) in Theorem \ref{tVoisin} are satisfied. Artin and Mumford \cite{AM72}, Proposition 3, have given such a criterion.

\begin{theorem}\xlabel{tArtinMumford}
Suppose $\pi \colon X \to S$ is a conic bundle over a smooth rational surface $S$, i.e. there exists a vector bundle $\mathcal{E}$ of rank $3$ on $S$, an integer $k$ and a quadratic form $q\in H^0 (S, \mathrm{Sym}^2 \mathcal{E} (k))$ such that $X$ is the zero scheme of $q$ in the associated projective bundle $\PP (\mathcal{E})$ over $S$, and $q$ is generically nondegenerate. Suppose moreover that $q$ is of corank $1$ everywhere above the curve $\Delta \subset S$ where it is not nondegenerate (so the fibers over the discriminant curve are two distinct lines everywhere), and that $\Delta$ consists of more than one smooth components, and these components $\Delta_i$ meet transversally. Over $\Delta$ we have a natural double cover $\tilde{\Delta} \to \Delta$, given by the subset in the Grassmannian of lines in the fibers of $\PP (\mathcal{E})$ consisting of lines contained in $X$. Suppose that $\tilde{\Delta}$ gives a nontrivial \'{e}tale double cover when restricted to any component $\Delta_i$ (equivalently, the two-torsion line bundle $\alpha$ determined on $\Delta$ by $\tilde{\Delta}$ restricts nontrivially to each component $\Delta_i$). 
Then $X$ has only ordinary double point singularities lying in fibers above points of $\Delta$ where two components meet, hence has a desingularization $\sigma \colon \tilde{X} \to X$ of the type required in (2) of Theorem \ref{tVoisin} above. Moreover, (1) of that Theorem holds for $\tilde{X}$.
\end{theorem}

Hence our strategy for proving Theorem \ref{tMain} is clear: we have to find an appropriate degeneration of the conic bundles given by our hypersurfaces $H_{2,n}$ to apply Theorem \ref{tArtinMumford}. Usually, this task is broken up into two steps:

\begin{itemize}
\item[Step 1.]
Prove that there is a degeneration of Prym curves 
\[
(\Delta_{H_{2,n, t}}, \alpha_t ) \longrightarrow (\Delta , \alpha ), \quad t\in B,\; t\to o\in B
\]
such that $(\Delta, \alpha)$ is of the type required for application of Theorem \ref{tArtinMumford}. 
\item[Step 2.]
Prove that there is a family of conic bundles $f \colon \tilde{X} \to B$ (to emphasize the conic bundle structure we might also write this as $\mathcal{C} \to \PP^2 \times B$) whose discriminant data realize the degeneration of Prym curves of Step 1, and such that  Theorem \ref{tVoisin} is applicable.
\end{itemize}

In \cite{HKT15} (with a slight generalization in \cite{HT16}) Hassett, Kresch and Tschinkel give a solution to Step 2 in the following way.

\begin{theorem}\xlabel{tHKT}
Let S a be a smooth projective rational surface, and let $\mathcal{P}$ be an irreducible variety parametrizing pairs $(C, \alpha )$ where
\begin{enumerate}
\item
the curves $C$ belong to some linear system of effective divisors on $S$, are smooth and irreducible for a generic point in $\mathcal{P}$ and reduced nodal in general;
\item
$\alpha$ is a $2$-torsion line bundle on $C$ which is nontrivial over each irreducible component of $C$.
\end{enumerate}
If then $\mathcal{P}$ contains a point $p_0=(C_0, \alpha_0)$ with $C_0$ a reducible curve with smooth irreducible components, then the very general conic bundle $\mathcal{C}_p \to S$ constructed from a point $p=(C, \alpha )$ in $\mathcal{P}$ is not stably rational.
\end{theorem}

Here $\mathcal{C}_p$ is defined up to birational isomorphism by the construction in \cite{AM72}, Theorem 1.

This, however, does not achieve Step 1 above at all (finding the appropriate degenerations of discriminant data tends to be the hardest part in many applications), and moreover, to prove Theorem \ref{tHKT}, the authors make use of deformation theory of tame Deligne-Mumford stacks, root stacks and sheaves of Azumaya algebras on them, which is very technically involved.

What this article accomplishes is the following:

\begin{itemize}
\item[(a)]
We find appropriate degenerations of Prym curves for Step 1 in for $n=2m$ even, and prove Theorem \ref{tMain} for even $n$ in this way. For the case of general $n$ we use a result due to Colliot-Th\'{e}l\`{e}ne and Totaro \cite[Lemma 2.4]{To16}. The latter possibility and the method of proof was kindly communicated to us by Zhiyu Tian after the first version of this article appeared online; the method is also used by Zhi Jiang, Zhiyu Tian, and Letao Zhang in a forthcoming work. The first version only proved Theorem \ref{tMain} for even $n$.
\item[(b)]
We do not use the results of \cite{HKT15} at all, but rather replace the deformation theory of stacks and sheaves of Azumaya algebras by a construction involving the Koszul complex; this is much easier and more concrete in the particular case we are interested in. It could also be used to investigate stratifications of the Prym moduli space and degenerations of Prym curves more systematically in future.
\item[(c)]
We construct several examples of reducible Prym curves such that Theorem \ref{tArtinMumford} is applicable to imply that the associated conic bundles are not stably rational.
\end{itemize}

It should be said that in \cite{HT16}, p. 16 bottom, the authors point out, without addressing the details, that the case $n=2$ of Theorem \ref{tMain} can be obtained as a corollary to their general theory of deformations of root stacks and sheaves of Azumaya algebras, but to obtain the statement for any $n$, it is necessary to resort to the constructions in the present article, and treat the case $n=2$ more explicitly as a starting point, too.

We add a few more words about compactifications of Prym moduli spaces to make clear why Step 1 above is nontrivial. The moduli space $\mathcal{R}_g$ of pairs $(C, \alpha)$ where $C$ is a smooth projective genus $g$ curve and $\alpha$ a two-torsion line bundle on $C$ admits a compactification $\bar{\mathcal{R}}_g$ which is compatible with the Deligne-Mumford compactification of the moduli space of curves of genus $g$ under the forgetful map; i.e. it extends to a morphism $\bar{\mathcal{R}}_g \to \bar{\mathcal{M}}_g$. See \cite{Beau77}, \cite{BCF04} or \cite{Fa12} for this. Recall that a curve is called (semi-)stable if it is a reduced connected one-dimensional schemes with at most ordinary double points and of arithmetic genus $g$ such that every smooth rational component $E$ meets the other components in $\ge 3$ (resp. $\ge 2$ for semi-stable curves) points. Then $\bar{\mathcal{M}}_g$ contains stable curves and is compact. Every one-parameter family of smooth curves has a limit in $\bar{\mathcal{M}}_g$, though one has to first perform a semi-stable reduction to see this (e.g. a family of plane cubics specializing to a cuspidal curve has a curve with an elliptic tail as limit, after several blow-ups in the central fiber and finite covers of the base).

Now by \cite{BCF04}, Definition 1, points in $\bar{\mathcal{R}}_g$ parametrize the following objects.

\begin{definition}\xlabel{dBallico}
A component $E$ of a Deligne-Mumford semi-stable curve $C$ is called \emph{exceptional} if it is smooth, rational and meets the other components in exactly two points. One calls $C$ \emph{quasi-stable} if every two exceptional components are disjoint. The stable model $\mathrm{st}(C)$ is the stable curve obtained from $C$ by contracting all exceptional components. 

A (semi-stable) Prym curve $C$ of genus $g$ is a triple $(C, \eta , \beta )$ where $C$ is a quasi-stable curve of genus $g$, $\eta$ is a line bundle on $C$ with a sheaf homomorphism $\beta \colon \eta^{\otimes 2} \to \mathcal{O}_C$ such that
\begin{enumerate}
\item
$\eta$ has total degree $0$ on $C$ and degree $1$ on every exceptional component;
\item
$\beta$ is non-zero at the general point of every non-exceptional component.
\end{enumerate}
\end{definition}

Equivalently, this means that $\beta$ vanishes identically on all exceptional components $E_i$ of $C$, and denoting by $\tilde{C}$ the union of the non-exceptional components
\[
\eta\mid_{\tilde{C}} \simeq \mathcal{O}_{\tilde{C}}(-q_1^1-q_1^2- \dots - q_r^1 -q_r^2)
\]
where $\tilde{C} \cap E_i = \{ q_i^1, q_i^2 \}$. Moreover, $\eta\mid_{E_i} = \mathcal{O}_{\PP^1}(1)$, and the map $\bar{\mathcal{R}}_g \to \bar{\mathcal{M}}_g$ is given by associating $\mathrm{st}(C)$ to the triple $(C, \eta , \beta )$. 

\medskip

Jarvis \cite{Jar98} has given an equivalent description of this compactification in terms of (square) root sheaves of $\mathcal{O}_{\mathrm{st}(C)}$ (certain rank one torsion free coherent sheaves on  $\mathrm{st}(C)$).

We give a name to the types of Prym curves that we allow as our degenerations in Step 1 above (they are those for which Theorem \ref{tArtinMumford} ensures a nonvanishing Brauer obstruction for the associated conic bundle).

\begin{definition}\xlabel{dGoodPrym}
We will call a (stable) plane Prym curve $(C, \alpha )$ \emph{good} if the following hold
\begin{enumerate}
\item
$C$ is reducible, with smooth irreducible components $C_i$, $i\in I$, and nodal.
\item
The torsion-free symmetric rank $1$ root sheaf $\alpha$ on $C$ is in fact a two torsion line bundle, and restricts to a nontrivial two torsion line bundle on each component $C_i$.
\end{enumerate}
\end{definition}

Now we can say more precisely what the subtlety of Step 1 above consists in: if we look at the closure of the locus of smooth plane Prym curves $(C, \alpha)$ such that the two-torsion line bundle $\alpha$ has a minimal graded free resolution of type \ref{fResAlpha} inside the respective $\bar{\mathcal{R}}_g$, then this closure need not contain any good Prym curves at all! 

\medskip

The subsequent sections are organized as follows: in Section \ref{sDetDeg} we prove Theorem \ref{tMain} for $n=2$ using determinantal degenerations. For this it is necessary to find a symmetric $3\times 3$ matrix with quadratic forms on $\PP^2$ as entries whose determinant defines a union of two smooth cubic plane curves intersecting transversally, and such that the corank of the matrix is precisely $1$ in each point of the two curves.
It is not at all clear how to produce such a matrix ``by hand" and ad hoc attempts to write one down fail. 

In Section \ref{sGoodDegPrym} we prove Theorem \ref{tMain} for any even $n=2m$ then. We do not use determinantal degenerations, but rather degenerate to certain Prym curves that have minimal graded free resolutions of a type first studied in \cite{AM72} where the entries of the presentation matrix have different degrees and the matrix itself is $2\times 2$ instead of $3\times 3$. Using a construction involving a bi-graded Koszul complex on $\PP^2 \times \PP^1$ we solve the problems mentioned in Steps 1 and 2 above at the same time for the case $n=2$, and then employ a combination of a trick first appearing in \cite{Beau00} and a generalization of a geometric construction of good Prym curves in \cite{AM72} to settle the general case $n=2m$ in Theorem \ref{tMain}. We believe that some of the techniques in this section could also be of independent interest, e.g., to construct standard conic bundles associated to Prym curves explicitly, or to study adjacency relations in compactified moduli spaces of Prym curves.

In Section \ref{sGeneralTheorem} we finally give the proof of Theorem \ref{tMain} in full, for any $n$ independent of the parity; the proof was communicated to us by Zhiyu Tian (the method is also used in forthcoming work of Zhi Jiang, Zhiyu Tian, and Letao Zhang) and reduces the statement to the $n=2$ case by an induction, which in turn is based on a result due to Colliot-Th\'{e}l\`{e}ne and Totaro \cite[Lemma 2.4]{To16}.

\thanks{\textbf{Acknowledgments.} We would like to thank Ivan Cheltsov for pointing out the question to us in the first place, and Fedor Bogomolov, Fabrizio Catanese, Jean-Louis Colliot-Th\'{e}l\`{e}ne, Andrew Kresch, Kristian Ranestad, Miles Reid, and Yuri Tschinkel for useful discussions and suggestions about part of the material in this article.

We are especially thankful to Zhi Jiang, Zhiyu Tian, and Letao Zhang for communicating the material in Section \ref{sGeneralTheorem} to us and letting us include it in a revised version of this text.}

\section{Determinantal degenerations}\xlabel{sDetDeg}

In this Section we prove the case $n=2$ of Theorem \ref{tMain} using determinantal degenerations. Note that in this case, $A(u,v,w)$ is a three by three symmetric matrix of quadratic forms on $\PP^2$, which is general with this property for a very general hypersurface $H_{2,n}$, so to apply Theorem \ref{tVoisin} in conjunction with Theorem \ref{tArtinMumford}, it suffices to prove a symmetric three by three matrix of quadratic forms with determinant the union of two smooth cubic curves meeting transversely, and such that the two-torsion line bundle $\alpha$ defined by this matrix restricts nontrivially to each of these two cubics.

We will first accomplish this over a finite field, and then show that our example lifts to characteristic $0$. 

We recall two results from \cite{Beau00} which we will use.

\begin{theorem}\xlabel{tBeau1}
Let $(C, \alpha )$ be a pair consisting of an integral plane curve $C$ of degree $d$ and a non-trivial line bundle $\alpha$ with $\alpha^{\otimes 2} \simeq \mathcal{O}_C$.
\begin{enumerate}
\item
If $d$ is even, $d=2e$, then for a \emph{general} pair $(C, \alpha )$ the line bundle $\alpha$ has a minimal resolution
\begin{gather}\label{fResGen}
\xymatrix{
0 \ar[r] & \mathcal{O}_{\PP^2}(-e-1)^e \ar[r]^M & \mathcal{O}_{\PP^2}(-e+1)^e \ar[r] & \alpha \ar[r] & 0
}
\end{gather}
with $M$ symmetric with quadratic entries, and $\det M$ is a defining equation for $C$.
\item
If $d$ is odd, then $\kappa := \alpha ((d-3)/2)$ is a theta-characteristic on $C$, i.e. $\kappa^{\otimes 2}\simeq \omega_C$. The moduli space of pairs $(C, \kappa )$ has two (non-special) components, according to whether $h^0(\kappa)$ is even or odd, and a general point in these components satisfies $h^0(\kappa ) \le 1$.
\begin{enumerate}
\item
If $C$ is smooth and $h^0(\kappa ) =0$, then $\kappa$ admits a minimal resolution
\begin{gather}
\xymatrix{
0 \ar[r] & \mathcal{O}_{\PP^2}(-2)^d \ar[r]^M & \mathcal{O}_{\PP^2}(-1)^d \ar[r] & \kappa \ar[r] & 0
}
\end{gather}
where $M$ is symmetric linear and $\det M$ defines $C$.
\item
If $h^0 (\kappa ) =1$ and $C$ is smooth, then $\kappa$ admits a minimal resolution
\begin{gather}
\xymatrix{
0 \ar[r] & \mathcal{O}_{\PP^2}(-2)^{d-3}\oplus \mathcal{O}_{\PP^2}(-3) \ar[r]^M & \mathcal{O}_{\PP^2}(-1)^{d-3}\oplus \mathcal{O}_{\PP^2} \ar[r] & \kappa \ar[r] & 0
}
\end{gather}
where $M$ is symmetric with homogeneous forms as entries and $\det M$ defines $C$. 
\end{enumerate}
\end{enumerate}
\end{theorem}

This is \cite{Beau00}, Prop. 4.2 and Prop. 4.6.

In our present case of plane sextics, $e=3$, and we would like to understand if we can find a reducible sextic, splitting as two cubics meeting transversely, with an $\alpha$ that has a resolution of the form (2.1). Using another result of Beauville, we can formulate a criterion for this.

\begin{theorem}\xlabel{tBeau2}
Let $(C, \alpha )$ be a pair consisting of an integral plane curve $C$ of degree $d=2e$ and a two-torsion line bundle $\alpha$ on $C$. Put $\mathcal{L}:= \alpha (e-1)$. Thus $\mathcal{L}$ has degree $g-1+e$. The following are equivalent:
\begin{enumerate}
\item
$H^0 (C, \mathcal{L}(-1))= H^1 (C, \mathcal{L}) =0$ and the multiplication map
\[
\mu_0\colon H^0(C, \mathcal{L}) \otimes H^0 (C, \mathcal{O}_C(1)) \to H^0( C, \mathcal{L}(1))
\]
is an isomorphism.
\item
There is an exact sequence
\begin{gather}
\xymatrix{
0 \ar[r] & \mathcal{O}_{\PP^2}(-2)^e \ar[r]^N & \mathcal{O}_{\PP^2}^e \ar[r] & \mathcal{L}\ar[r] & 0
}
\end{gather}
and $\det N$ defines $C$. Moreover, if (1) or (2) holds, one can symmetrize the resolution, i.e. one can choose $N$ symmetric. 
\end{enumerate}
\end{theorem}

The first part about the equivalence of (1) and (2) is \cite{Beau00}, Prop. 3.5 in this special case, and the symmetrization statement is \cite{Beau00}, Theorem B.

\medskip

Let us now focus on the case $e=3$ and plane sextics again. Our first task is to describe two-torsion line bundles $\alpha$ sitting in a resolution (2.1) concretely. Such an $\alpha$ must satisfy that $\mathcal{L}=\alpha (2)$ has three sections. It is hence of the form $\mathcal{L}=\mathcal{O}(D)$ for a certain effective divisor $D$ of degree $12$ on $C$. We can assume that $D$ consists of $12$ points. There is a quartic containing these $12$ points, cutting out a divisor $D'$ residual to $D$ on $C$ where $\deg D' =12$ as well, and $|D| = |4H -D'|$, thus the linear system corresponding to $D$ is cut out by the quartics through $D'$. Actually, since $2D \equiv 4H$ as $\mathcal{L}^{\otimes 2} \simeq \mathcal{O}_C(4)$, one can choose $D'=D$, too. 

The equation $2D \equiv 4H$ means that there should be a quartic which cuts out the points in $D$ on $C$ with a double structure on $C$, hence is tangent to $C$ in those points. This condition is equivalent to $2D\equiv 4H$, hence to $\alpha$ being two-torsion. 

Thus our construction of a pair $(C, \alpha )$ with $C=C_1 \cup C_2$ splitting as two smooth cubics meeting transversely and $\alpha$ nontrivial $2$-torsion on $C$, $h^0(C, \alpha (2))=3$, having a resolution as in (2.1), proceeds via the following steps. We work over a finite field $\FF$ first, then discuss the lifting problem to characteristic $0$.

\begin{enumerate}
\item
Pick a smooth cubic $C_1$ and six points $D_1:=P_1\cup \dots \cup P_6$ on $C_1$ at random. 
\item
Compute the ideal of $2D_1$ (double on $C$!) and check whether the element of smallest degree in it is a quartic. This is not always the case, so if not, go back to Step 1. and wait till you get a quartic $Q_4$. This works because the codimension of the parameter space of the sought-for pairs $(C_1, D_1)$ is not too high.
\item
Pick six points $D_2:=Q_1\cup \dots \cup Q_6$ on $Q_4$ at random (different from $D_1$) and search for a cubic $C_2$ tangent to $Q_4$ in $D_2$ by the same random procedure as in 1. and 2.
\item
Take $C=C_1\cup C_2$ and $\mathcal{L}:= \mathcal{O}_C(4H -D_1-D_2)$. Check whether now $h^0(C, \mathcal{L})=3$, i.e. whether there is a three-dimensional space of quartics through $D=D_1\cup D_2$. This is not always the case (e.g. sometimes one gets a $4$-dimensional space), in which case one repeats the entire procedure until one succeeds.
\item
Compute a resolution of the full module of sections of $\mathcal{L}$ on $C$. This gives you a matrix $N$ as in (2.4).
\item
Symmetrize $N$ in the following way: by the symmetrization statement in Theorem \ref{tBeau2}, there will be scalar base change matrices $A$ and $B$ such that $ANB$ is symmetric; but then also $(A^{-1})AN B (A^{-1})^t$ will be symmetric. In other words, there will be a matrix $S$ such that $NS=:M$ is symmetric. These are linear equations for the entries of $S$ which are easy to solve.
\end{enumerate}

Carrying out steps 1 to 6, we found the following:
\begin{proposition}\xlabel{pDetDeg}
Consider the matrix 
{\tiny
\begin{gather*}
M =\\
 \begin{pmatrix} 
      -10x^2+xy-8y^2+8xz+5yz-9z^2 &
       -4x^2-5xy+3y^2+5xz-11yz-7z^2 &
       4xy-8y^2+6xz+yz-8z^2 \\
      -4x^2-5xy+3y^2+5xz-11yz-7z^2 &
       -8x^2+9xy-2y^2-7xz+yz-9z^2 &
       8xy-5y^2-6xz+11yz+9z^2\\
      4xy-8y^2+6xz+yz-8z^2 & 
       8xy-5y^2-6xz+11yz+9z^2 & 
       xy-6y^2+10xz+2yz+2z^2
  \end{pmatrix}
\end{gather*}
}
%M =\\
 %\begin{pmatrix} 
 %     -10x_0^2+x_0x_1-8x_1^2+8x_0x_2+5x_1x_2-9x_2^2 &
  %     -4x_0^2-5x_0x_1+3x_1^2+5x_0x_2-11x_1x_2-7x_2^2 &
%       4x_0x_1-8x_1^2+6x_0x_2+x_1x_2-8x_2^2 \\
%      -4x_0^2-5x_0x_1+3x_1^2+5x_0x_2-11x_1x_2-7x_2^2 &
%       -8x_0^2+9x_0x_1-2x_1^2-7x_0x_2+x_1x_2-9x_2^2 &
%       8x_0x_1-5x_1^2-6x_0x_2+11x_1x_2+9x_2^2\\
%      4x_0x_1-8x_1^2+6x_0x_2+x_1x_2-8x_2^2 & 
%       8x_0x_1-5x_1^2-6x_0x_2+11x_1x_2+9x_2^2 & 
%       x_0x_1-6x_1^2+10x_0x_2+2x_1x_2+2x_2^2
%  \end{pmatrix}
%\end{gather*}
over the finite field $\FF_{23}$. Then:
\begin{enumerate}
\item
The determinant $\det M$ defines two smooth cubic curves $C_1$ and $C_2$ meeting transversely, and $M$ has corank $1$ in every point of $C_1\cup C_2$.
\item
The two torsion line bundle $\alpha$ defined by $M$ is nontrivial on both $C_1$ and $C_2$.
\end{enumerate}
\end{proposition}

\begin{proof}
This is a Macaulay2 computation, see \cite{BB-M2-16}.
\end{proof}

Now we have to address the lifting problem.

\begin{lemma}\xlabel{lLifting}
Suppose that over a finite field $\FF_p$, $p\neq 2$, there is a reducible integral curve $C$ which splits as a union of two smooth cubics $C=C_1\cup C_2$ meeting transversely, and a nontrivial two-torsion line bundle $\alpha$ on $C$ which has a resolution of the form (2.1); equivalently, this amounts to the existence of a symmetric matrix $M$ with quadratic forms as entries such that $\det M=0$ defines two smooth cubics meeting transversely such that $M$ has rank $2$ in every point of $C=C_1\cup C_2$. Also assume that $\alpha$ restricts nontrivially to both $C_1$ and $C_2$. Then there exists such an example over $\CC$ as well. 
\end{lemma}

\begin{proof}
The parameter space of the pairs $C_1, C_2$ is given by 
\[
P:=\PP ( \FF_p [X_0, X_1, X_2]_3)^0 \times \PP (\FF_p [X_0, X_1, X_2]_3)^0
\]
a product of two open subsets in projective spaces (already defined over $\ZZ$), and the datum of an $\alpha$ on $C_1\cup C_2$ amounts to the datum of a nontrivial two-torsion line bundle $\alpha_1$ on $C_1$, a nontrivial two-torsion line bundle on $C_2$, and an isomorphism of the vector bundle fibers $(\alpha_1)_x \to (\alpha_2)_x$ in each intersection point of $x\in C_1\cap C_2$. Now the space $P$ is the special fiber, over the closed point corresponding to $p$, of a family
\[
\mathcal{P} \to \mathrm{Spec}(\ZZ )
\]
 where $\mathcal{P}$ is an open in a product of two projective spaces defined over $\ZZ$; and by \cite{Ta97}  the two-torsion points in the product of relative Jacobians over $\mathcal{P}$ form a finite flat group scheme since $p\neq 2$. Hence we can lift $C_1\cup C_2$ as well as $\alpha_1$ and $\alpha_2$ to characteristic zero, and we can also lift the linear isomorphisms $(\alpha_1)_x \to (\alpha_2)_x$ in the intersection points along with this.

Hence we get a family defined over some ring of integers $\mathfrak{o}$ in a number field. Now by \ref{tBeau2}, the property that $\alpha$ on $C$ has a resolution of the required type is generic in this family, and it holds at a closed point of $\mathrm{Spec}(\mathfrak{o})$ lying over $p$. Hence it holds at the generic point, hence over $\CC$ as well, thus for a lift to characteristic $0$. Also the lift of $\alpha$ will restrict nontrivially to each component since this is true at the special point. 
\end{proof}

\begin{theorem}\xlabel{tNonStabRat22}
The very general hypersurface of bidegree $(2,2)$ in $\PP^2\times \PP^2$ is not stably rational (over $\CC$).
\end{theorem}

\begin{proof}
This follows immediately from Theorems \ref{tVoisin} and \ref{tArtinMumford} now that we have constructed the matrix $M$ above and proven Lemma \ref{lLifting}. 
\end{proof}

If one wants to proceed further, i.e. treat cases of hypersurfaces of bidegree $(2,n)$, $n$ even, for $n>2$, then the above brute-force computational approach does not work anymore (the codimension of the sought-for Prym curves in their parameter space is too high). Thus in the next section we use a different method to prove Theorem \ref{tMain} for any even $n$.

\section{Construction methods for good Prym curves and degenerations to Artin-Mumford type examples via the Koszul complex}\xlabel{sGoodDegPrym}

In this Section we prove Theorem \ref{tMain} for any $n=2m$ even. We start with a geometric method to construct good plane Prym curves in the sense of Definition \ref{dGoodPrym}, generalizing a geometric construction of Artin and Mumford in \cite{AM72}  p. 79 ff. and p. 93/94.

\begin{lemma}\xlabel{lAM}
Let $A$ be a smooth plane curve of even degree $a$, and let $C_1, \dots , C_r$ be smooth plane curves of degrees $c_1, \dots , c_r$ where all $c_i > a$, and every $C_i$ is tangent to $A$ in $c_ia/2$ distinct points. Let $\mathcal{P}$ be this set of points on $C=C_1\cup \dots \cup C_r$, and let $\mathcal{P}_i$ be the set of points lying on $C_i$. Then the line bundle $\alpha$ associated to $\mathcal{P} - (a/2) h$ (where $h$ is the intersection of $C$ with a general line in $\PP^2$) is nontrivial $2$-torsion on every component $C_i$.
\end{lemma}

\begin{proof}
This is an easy extension of \cite{AM72}, Lemma p. 93: by construction, $2\mathcal{P}$ is cut out by $A$ on $C$, so $\alpha$ is $2$-torsion; it is nontrivial on every $C_i$ by the following reasoning: suppose by contradiction that it was trivial on $C_i$. Let $q$ be the rational function on $\PP^2$ whose divisor is $A - aH$, $H$ some fixed line in $\PP^2$ so that $h=H\cap C$. Then the restriction $\bar{q}$ of $q$ to $C_i$ would be a square $\bar{q} = \bar{s}^2$ for some $\bar{s}\in \CC (C_i)$ if $\alpha$ was trivial on $C_i$. Then 
\[
\bar{s} \in H^0 (C, \mathcal{O}_{C_i}((a/2)H . C_i)
\]
and since $H^0 (\PP^2, \mathcal{O}_{\PP^2}((a/2) H)) \to H^0 (C_i, \mathcal{O}_{C_i}((a/2) H . C_i)$ is surjective, there would be a function $s$ in $\CC (\PP^2)$ that lifts $\bar{s}$. Moreover, $(s) = R - (a/2)H$ where $R$ is another curve of degree $a/2$ in $\PP^2$. But then, set-theoretically, $A\cap C_i$ would be equal to $R \cap C_i$, hence $A \cap C_i \subset A \cap R$, and the latter consists of at most $a^2/2$ points, which is strictly smaller than $ac_i/2$, contradiction.
\end{proof}

\begin{remark}\xlabel{rRanestad}
In particular, if the curve $C = C_1 \cup \dots \cup C_r$ in Lemma \ref{lAM} has only ordinary double points, the the pair $(C, \alpha )$ is a good Prym curve and by Theorem \ref{tArtinMumford}, the associated conic bundle is not stably rational. Here is one geometric way to produce an arrangement of curves $A, C_1, \dots , C_r$ as in Lemma \ref{lAM}: in fact, it suffices to show how to produce smooth plane curves $C, A$, $\deg C > \deg A$ with $\deg A = a = 2a'$ even, and $C$ tangent to $A$ in $a'c$ distinct points. For this, we can use  Theorem \ref{tBeau1} (1) and Bertini's theorem. Start with a smooth $A$ of degree $a=2a' >2$ with a nontrivial square-root $\alpha$ of $\mathcal{O}_A$ with a resolution of type \ref{fResGen}. Suppose the linear system $|\alpha (m)|$ is of dimension $\ge 1$ and base-point free. By Theorem \ref{tBeau1} (1) this will be the case as soon as $m \ge a' -1$. Hence, by Bertini's base-point free pencil theorem, there will be an effective divisor in $|\alpha (m)|$, $D$ say, consisting of $m\cdot a$ distinct points. Then $|2D| = |\mathcal{O}(2m)|$, hence there is a curve $C_0$ of degree $c= 2m$ intersecting $A$ tangentially in the points in $D$. Suppose now we choose $m$ such that $c > a$. Consider the linear system $\mathcal{L}$ of plane curves of degree $c$ intersecting $A$ tangentially in the points in $D$. $\mathcal{L}$ contains curves which are smooth in the points in $D$, e.g. $A + C'$ for a general curve $C'$ of degree $c-a$. This also shows that the linear system $\mathcal{L}$ has no basepoints outside of $D$ (since $C_0$ intersects $A$ only in $D$ and belongs to $\mathcal{L}$, but a curve of type $A + C'$ can avoid any given point outside $A$). Hence the general curve in $\mathcal{L}$ will be smooth (by Bertini's theorem again, and this will also hold in the points $D$ since there are curves in $\mathcal{L}$ that are smooth in $D$) and cut out precisely $2D$ on $A$ as desired. 

This method produces many examples of good Prym curves and hence of conic bundles which are not stably rational. 

We would like to thank Kristian Ranestad for suggesting this geometric approach. 
\end{remark}

For our immediate purpose of proving Theorem \ref{tMain} for $n$ even, we will, however, construct good Prym curves of degree $\deg C = d$ even and divisible by $3$, hence $d = 6m$, for $m\ge 1$, by a ``dirty trick" different from the construction method in Remark \ref{rRanestad}. Note that $n=2m$ then in the notation of Theorem \ref{tMain}. These curves will have minimal graded free resolutions
\begin{gather}\label{fPrymGood}
\xymatrix{
0 \ar[r] & \mathcal{O}_{\PP^2}\left( -5m \right)\oplus\mathcal{O}_{\PP^2}\left( -4m \right) \ar[r]^M & \mathcal{O}_{\PP^2}\left(-m \right) \oplus \mathcal{O}_{\PP^2}\left(-2m \right) \ar[r] & \alpha \ar[r] & 0
}
\end{gather}
where $M$ is a two by two matrix with entries homogeneous polynomials in $\CC [u, v, w]$ of degrees
\begin{gather*}
\begin{pmatrix}
4m & 3m\\
3m & 2m
\end{pmatrix}.
\end{gather*}
The construction method will allow us to conclude that we can degenerate our discriminant Prym curves with resolution type \ref{fResAlpha} to these good Prym curves once we have proven that fact for $n=2$.

\begin{proposition}\xlabel{pGoodPryms}
There are good Prym curves $(C', \alpha')$ with resolution type \ref{fPrymGood} which are pull-backs of good plane sextic Prym curves with resolution type \ref{fPrymGood} (for the case $m=1$) under a (degree $m^2$ ramified) covering map
\[
\gamma\colon \PP^2 \to \PP^2 , \quad (u:v:w) \mapsto (u^m : v^m : w^m).
\]
\end{proposition}

\begin{proof}
It is easy to construct good Prym curves of degree $6$ using Lemma \ref{lAM} with $A$ a conic and $C= C_1 \cup C_2$ with $C_i$ cubics: this is what Artin and Mumford do on page 79 ff.; it just amounts to the existence of smooth cubics tangent to a given conic in some set of three points (which we are free to choose a priori). Now if we choose coordinates $u, v, w$ such that $A, C_1, C_2$ and the cubic cutting out the six tangency points on $A$ are all transverse to $u=0$, $v=0$ and $w=0$ in smooth points, and contain none of the intersection points of two of the coordinate axes, then we can apply Lemma \ref{lAM} to the curves $A' = \gamma^* (A), \gamma^* (C_1), \gamma^* (C_2)$ (all of these are smooth, hence irreducible, under the above assumptions) to conclude that $\alpha' = \gamma^* (\alpha )$ is still nontrivial on every component of $C' = \gamma^* (C)$ (this is the main usage of Lemma \ref{lAM}) and $(C', \alpha')$ has resolution type \ref{fPrymGood} with a presentation matrix
\begin{gather*}
\begin{pmatrix}
c' & b'\\
b' & a'
\end{pmatrix}
\end{gather*}
where $a'$ has degree $2m$, $c'$ degree $4m$, $b'$ degree $3m$, and this matrix is obtained from a matrix
\begin{gather*}
\begin{pmatrix}
c & b\\
b & a
\end{pmatrix}
\end{gather*}
where $a$ has degree $2$, $c$ degree $4$, $b$ degree $3$, by the substitution 
\[
u\mapsto u^m, v\mapsto v^m, w \mapsto w^m.
\]
\end{proof}

To prove Theorem \ref{tMain} for $n$ even we now start with a solution of the case $n=2$ (where the discriminant curve is a plane sextic) that is different from the one in the previous section, and involves degeneration to Artin-Mumford type Prym sextics. A few more pieces of terminology are useful.

\begin{definition}\xlabel{dAMGenSextic}
Henceforth in this section a plane sextic Prym curve $(C, \alpha)$ will mean an at most nodal reduced plane curve $C$ with smooth irreducible components where $\alpha$ is a two torsion line bundle on it. 

We call $(C, \alpha)$ \emph{of general type} if $C$ is smooth, and $\alpha$ is nontrivial with minimal graded free resolution
\begin{gather}\label{fSexticGeneral}
\xymatrix{
0 \ar[r] & \mathcal{O}_{\PP^2}(-4)^3 \ar[r]^A & \mathcal{O}_{\PP^2}(-2)^3 \ar[r] & \alpha \ar[r] & 0
}
\end{gather}
where $A$ is a symmetric three by three matrix with quadratic entries.

We call a plane sextic Prym curve $(C, \alpha)$ \emph{of Artin-Mumford type} if there is a minimal graded free resolution
\begin{gather}\label{fSexticGeneral}
\xymatrix{
0 \ar[r] & \mathcal{O}_{\PP^2}(-4)^3 \ar[r]^A & \mathcal{O}_{\PP^2}(-2)^3 \ar[r] & \alpha \ar[r] & 0
}
\end{gather}
where $B$ is symmetric and the degrees of the entries in $B$ are
\[
\begin{pmatrix}
4 & 3\\
3 & 2
\end{pmatrix}.
\]
We call a plane Prym sextic curve of Artin-Mumford type good if moreover it is good in the sense of Definition \ref{dGoodPrym}. 
\end{definition}

The proof of Theorem \ref{tMain} for $n$ even will be an immediate consequence of Theorem \ref{tVoisin} and Theorem \ref{tArtinMumford} together with Theorem \ref{tSexticDegeneration} below and Proposition \ref{pGoodPryms} above.

\begin{theorem}\xlabel{tSexticDegeneration}
\begin{enumerate}
\item[(i)]
There is a Zariski-open neighborhood $T$ of the origin $0\in \mathbb{A}^1\subset \PP^1$, and a (flat) family of sextic plane Prym curves
\begin{gather}\label{fPrymFamily}
\xymatrix{
(\mathcal{C}, \underline{\alpha}) \ar[d]^{\pi} \\
B 
}
\end{gather}
such that for $t\in B$, $t\neq 0$, the fiber $(\mathcal{C}_t, (\underline{\alpha})_t)$ is a Prym sextic of general type, and $(\mathcal{C}_0, (\underline{\alpha})_0)$ is a plane sextic Prym curve of Artin-Mumford type.
\item[(ii)]
Any general sextic plane Prym curve of Artin-Mumford type occurs as the central fiber $(\mathcal{C}_0, (\underline{\alpha})_0)$ in a family as in (i).
\item[(iii)]
The families of Prym curves in (i) and (ii) can be chosen to arise as the family of discriminant Prym curves of a family of conic bundles over $\PP^2$
\begin{gather}\label{fConicBundle}
\xymatrix{
\mathcal{Q}\ar[d]^{p}\ar@{^{(}->}[r] &  \PP (\mathcal{E}) \ar[ld]^p\\
\PP^2 \times B & 
}
\end{gather}
where: 
\begin{itemize}
\item[(a)]
$\mathcal{E}$ is a rank $3$ vector bundle over $\PP^2\times B$ with $\mathcal{E}_t\simeq \mathcal{O}(-2)^3$ for $t\neq 0$ and $\mathcal{E}_0 \simeq \Omega^1(-1) \oplus \mathcal{O}(-1)$.
\item[(b)]
For $t\neq 0$ general, the total space of $\mathcal{Q}_t \to \PP^2 \times \{ t\}\simeq \PP^2$ is smooth and this conic bundle has discriminant Prym curve $(\mathcal{C}_t, (\underline{\alpha})_t)$.
\item[(c)]
For $t=0$, the total space of $\mathcal{Q}_0 \to \PP^2 \times \{ 0\}\simeq \PP^2$ has at worst double points as singularities and this conic bundle has discriminant Prym curve $(\mathcal{C}_0, (\underline{\alpha})_0)$. Double points occur if $(\mathcal{C}_0, (\underline{\alpha})_0)$ is good, hence does have nodes.
\end{itemize}
\item[(iv)]
There exists a family of conic bundles as in (iii) such that $(\mathcal{C}_0, (\underline{\alpha})_0)$ is a good sextic plane Prym curve of Artin-Mumford type.
\end{enumerate}
\end{theorem}

Before embarking on the proof of Theorem \ref{tSexticDegeneration}, let us show how Theorem \ref{tMain} for $n$ even follows from it.

\begin{proof}[Proof of Theorem \ref{tMain} for $n$ even:]
Using Theorem \ref{tArtinMumford}, Theorem \ref{tVoisin} is applicable to a family of conic bundles as in Theorem \ref{tSexticDegeneration} if the central fiber is a good Artin-Mumford plane sextic Prym curve, which we may assume by part (iv) of Theorem \ref{tSexticDegeneration}. This shows the case $n=2$ of Theorem \ref{tMain}.

If $n=2m>2$, we start with a family as for the proof of the case $n=2$ we just gave, and pull it back via the covering map of Proposition \ref{pGoodPryms}:
\[
\gamma\colon \PP^2 \to \PP^2 , \quad (u:v:w) \mapsto (u^m : v^m : w^m).
\]
If we choose the projective coordinate system generic, we will get a family of conic bundles with generic discriminant Prym curves of type \ref{fResAlpha}, and special fiber with discriminant Prym curve of type \ref{fPrymGood} and moreover good in the sense of Definition \ref{dGoodPrym} by Proposition \ref{pGoodPryms}. Hence Theorem \ref{tMain} for $n$ even follows from Theorem \ref{tVoisin} again. 
\end{proof}

Our method to prove Theorem \ref{tSexticDegeneration} is closely related to the Koszul complex and replaces the deformation theory of Azumaya algebras and tame Deligne-Mumford stacks in \cite{HKT15} by something a lot more concrete in this special situation. 

\emph{Description of the family of vector bundles in Theorem \ref{tSexticDegeneration}, (iii), (a):} By the Euler sequence
\begin{gather}\label{fEuler}
\xymatrix{
0 \ar[r] & \mathcal{O}_{\PP^2} \ar[r] & \mathcal{O}_{\PP^2}(1)^3 \ar[r] & \mathcal{T}_{\PP^2} \ar[r] & 0
}
\end{gather}
we see that we can take $\mathcal{E}_0^{\vee}(-6) :=\mathcal{O}(-5) \oplus \mathcal{T}(-5)$ as a bundle that  deforms to $\mathcal{O}(-4)^3$ (and is not graded free) because generically we have a non split extension giving $\mathcal{O}(-4)^3$, and the split extension is $\mathcal{E}_0^{\vee}(-6)$. Hence, $\mathcal{E}_0 \simeq \Omega^1 (-1) \oplus \mathcal{O}(-1)$. Also note the duality
\begin{gather}\label{fDualTang}
\mathcal{T} \simeq \Omega^1 (3)
\end{gather}
obtained from the nondegenerate pairing $\mathcal{T} \otimes \mathcal{T} \to K_{\PP^2}^{-1}$. 

In principle, this describes a family $\mathcal{E}$, but it is convenient to give a more explicit construction of $\mathcal{E}$ via the Koszul complex: take $\PP^2$ with homogeneous coordinates $u,v,w$ and $\PP^1$ with homogeneous coordinates $s, t$ and look at the bi-graded Koszul complex for the regular sequence $(u,v,w,t)$:
\begin{gather}\label{fKoszulComplex}
\xymatrix{
0 \ar[r] & \mathcal{O}_{\PP^2 \times \PP^1} \ar[r]^{\mathcal{S}\quad\quad\quad} & 3 \mathcal{O}(1, 0) \oplus \mathcal{O}(0,1) \ar[r]^{\mathcal{A}} & 3 \mathcal{O}(2, 0) \oplus 3 \mathcal{O}(1,1) \\
            &   \ar[r]^{\mathcal{A}^t\quad\quad\quad}   & \mathcal{O}(3, 0) \oplus 3 \mathcal{O}(2,1) \ar[r]^{\mathcal{S}^t} & \mathcal{O}(3,1) \ar[r] & 0 
}
\end{gather}
Here we write $\mathcal{O}(i,j)= \mathrm{pr}_{\PP^2}^*\mathcal{O}_{\PP^2}(i) \otimes \mathrm{pr}_{\PP^1}^*\mathcal{O}_{\PP^1}(j)$. 

We twist this complex by $(-5,0)$ and give identify the kernel of $\mathcal{A}^t$ and cokernel of $\mathcal{A}$:

\begin{gather}\label{fKoszulComplexTwisted}
\xymatrix{
            &                                                                                       &                                                             0    \ar[rd]                           &                                                         &  0\\
            &                                                                                       &                                                                                                        & \mathcal{E}^{\vee}(-6,0) \ar[rd]\ar[ru] &              \\ 
0 \ar[r] & \mathcal{O}(-5,0) \ar[r]^{\mathcal{S}\quad\quad\quad} & 3 \mathcal{O}(-4, 0) \oplus \mathcal{O}(-5,1) \ar[rr]^{\mathcal{A}}\ar[ru] &  & 3 \mathcal{O}(-3, 0) \oplus 3 \mathcal{O}(-4,1) \\
             \ar[rd]  \ar[rr]^{\mathcal{A}^t\quad\quad\quad}   &   &  \mathcal{O}(-2, 0) \oplus 3 \mathcal{O}(-3,1) \ar[r]^{\quad\quad \mathcal{S}^t} & \mathcal{O}(-2,1) \ar[r] & 0 \\
                                                                             & \mathcal{E}(-1, 1) \ar[ru] \ar[rd]       &     &   &  \\
                                                          0 \ar[ru]       &                                                    & 0  &  &
}
\end{gather}

Note that for $t\neq 0$, $\mathcal{E}^{\vee}(-6,k)$ restricts to $\mathcal{O}_{\PP^2}(-4)$ on $\PP^2 \times \{ (s:t)\}$, and for $t=0$, it restricts to $\mathcal{E}_0^{\vee}(-6)$ by the Euler sequence. 

\

\emph{Deformation theory of symmetric maps $\Phi_0 \colon \mathcal{E}_0^{\vee}(-6) \to \mathcal{E}_0$ to symmetric maps $\Phi \colon \mathcal{E}^{\vee}(-6, -2) \to \mathcal{E}$:} We will prove parts (ii), (iii), (iv) of Theorem \ref{tSexticDegeneration} in one stroke; we need a few preliminary observations.  We start with a method to produce symmetric such maps $\Phi_0$. 

Consider a diagram
\begin{gather}\label{dMNzero}
\xymatrix{
	0 \ar[r] & \sO(-5) \ar[r]^-{(u,v,w,0)^t} & 3\sO(-4) \oplus \sO(-5)  \ar[r]^-{A^t}  \ar[d]^{N_0} & 3\sO (-3) \oplus \sO(-5) \oplus \sO(-4)  \ar[d]^M & \\
	0 & \sO(-1) \ar[l] & 3\sO(-2) \oplus \sO(-1)  \ar[l]_-{(u,v,w,0)} & 3\sO (-3) \oplus \sO(-1) \oplus \sO(-2)    \ar[l]^-A  \\
	}
\end{gather}
where
\begin{gather}\label{fM}
	M = \begin{pmatrix}
	0 & 0 & u & v & w \\
	0 & f & c_u & c_v & c_w \\
	u & c_u & 1 & 0 & 0 \\
	v & c_v & 0 & 1 & 0 \\
	w & c_w & 0 & 0 & 1 
	\end{pmatrix}
\end{gather}

with $c_u, c_v, c_w$ and $f$ homogeneous polynomials of degrees $2,2,2$ and $4$ in $K[u,v,w]$, and
\[
	A = \begin{pmatrix}
	 0 & 0 & S\\ 
	 0 & 1 & 0
	 \end{pmatrix}
\]
with
\[
	S = \begin{pmatrix}
	  0 & w & -v \\
	  -w & 0 & u \\
	  v & -u & 0 \\
	  \end{pmatrix}.
\]
Every such 
\[
	N_0 = A M A^t
\]
defines a symmetric map $\Phi_0 \colon \sE_0^{\vee}(-6) \to \sE_0$ via
\begin{gather}\label{dPhizero}
\xymatrix{
	0 \ar[r] & \sO(-5) \ar[r]^-{(u,v,w,0)^t} & 3\sO(-4) \oplus \sO(-5)  \ar[r] \ar[d]^{N_0} & \sE_0^{\vee}(-6)  \ar[d]^{\Phi_0} \ar[r] & 0\\
	0 & \sO(-1) \ar[l] & 3\sO(-2) \oplus \sO(-1)  \ar[l]_-{(u,v,w,0)} & \sE_0    \ar[l] & \ar[l] 0 \\
	}
\end{gather}

\medskip 

Now we look for a deformation $\Phi$ of $\Phi_0$ 
\begin{gather}\label{dPhi}
\xymatrix{
	0 \ar[r] & \sO(-5,-2) \ar[r]^-{(u,v,w,t)^t} & 3\sO(-4,-2) \oplus \sO(-5,-1)  \ar[r]  \ar[d]^N & \sE^{\vee}(-6,-2) \ar[r] \ar[d]^\Phi & 0 \\
	0 & \sO(-1,0) \ar[l]  & 3\sO(-2,0) \oplus \sO(-1,-1)  \ar[l]_-{(u,v,w,t)} & \sE  \ar[l] & \ar[l] 0 \\
	}
\end{gather}
with $N_{\{t=0\}} = N_0$ and $\Phi_{\{t=0\}} = \Phi_0$. 

\begin{proposition}\xlabel{pDeformationPhi}
Every symmetric map $\Phi_0 \colon \sE_0^{\vee} (-6) \to \sE_0$ defined via a matrix $M$ as in \ref{fM} via diagrams \ref{dMNzero}  and \ref{dPhizero} can be deformed to a symmetric map $\Phi \colon \mathcal{E}^{\vee}(-6,-2) \to \mathcal{E}$ as in diagram \ref{dPhi}. 
\end{proposition}

\begin{proof}
For this observe that $N_0$ can be written as
\[
	N_0 = \begin{pmatrix} 
	-S^2 & g^t \\
		g & f
	       \end{pmatrix} 
\]
with $g = (g_u,g_v,g_w)$ a vector of degree $3$ polynomials and $g (u,v,w)^t = 0$.

We look for a symmetric $N$ over $K[u,v,w] \otimes K[s,t]$ such that
\[
 	N (u,v,w,t)^t = 0
\]
and
\[
	N_{\{t=0\}} = N_0.
\]
For this we write 
\[
	N_0 = 
	\begin{pmatrix} -S^2 & 0 \\ 0 & 0 \end{pmatrix}
	+ \begin{pmatrix} 0 & g^t \\ g & 0 \end{pmatrix}
	+ \begin{pmatrix} 0 & 0 \\ 0 & f \end{pmatrix}
\]
The first matrix already satisfies our conditions. We now show that there are symmetric 
matrices $G$ and $F$, correctly bi-graded, that reduce to the second and third matrix for $t=0$ and
also satisfy 
\[
	G (u,v,w,t)^t = F(u,v,w,t)^t =0.
\]
We then set
\[
	N = s^2\begin{pmatrix} -S^2 & 0 \\ 0 & 0 \end{pmatrix} + sG + F.
\]

Thus the proof of Proposition \ref{pDeformationPhi} is concluded by Lemmata \ref{lBo1} and \ref{lBo2} below.
\end{proof}

\begin{lemma}\xlabel{lBo1}
Let $g  = (g_u,g_v,g_w)$ be a vector of homogeneous polynomials of degree $3$ in $K[u,v,w]$ such that 
\[
	ug_u + vg_v + wg_w = 0.
\]
Then there exists a symmetric $4\times 4$ matrix $G$ over $K[u,v,w,t]$ such that
\[
	G_{\{t=0\}} = \begin{pmatrix} 0 & g^t \\ g & 0 \end{pmatrix}
\]
and
\[
	G \begin{pmatrix} u \\ v \\ w \\ t \end{pmatrix} = 0.
\]
Moreover, the entries of $sG$ in the upper left $3\times 3$ submatrix have bi-degree $(2,2)$ in $(u,v,w; s,t)$, the bottom right entry has bi-degree $(4,0)$, the remaining entries bi-degree $(3,1)$. 
\end{lemma}

\begin{proof}
Let
\[
	J = \begin{pmatrix}
	(g_u)_u & (g_u)_v & (g_u)_w \\
	(g_v)_u & (g_v)_v & (g_v)_w \\
	(g_w)_u & (g_w)_v & (g_w)_w 
	\end{pmatrix}
\]
be the Jacobian matrix of $g$. Observe that 
\[
	J \begin{pmatrix} u \\ v \\ w \end{pmatrix} = 3 g^t.
\]
Differentiating the equation
\[
	ug_u + vg_v + wg_w = 0
\]
we obtain
\[
	 g_u + u(g_u)_u + v(g_v)_u + w(g_w)_u  = 0
	 \iff u(g_u)_u + v(g_v)_u + w(g_w)_u = -g_u.
\]
as well as similar equations for $-g_v$ and $-g_w$.
It follows that
\[
	( u,v,w ) J = - g.
\]
Setting
\[
	G =  \begin{pmatrix}
	\frac{-t}{2}\bigl(J+J^t\bigr) & g^t \\
	g & 0
	\end{pmatrix}
\]
the above calculations show
\[
	G \begin{pmatrix} u \\ v \\ w \\ t \end{pmatrix} = 0.
\]
\end{proof}

\begin{lemma}\xlabel{lBo2}
Let $f \in K[u,v,w]$ be as above a  homogeneous polynomial of degree $4$. 
Then there exists a symmetric $4\times4$ matrix $F$ over $K[u,v,w,t]$ such that
\[
	F_{\{t=0\}} = \begin{pmatrix} 0 & 0 \\ 0  & f \end{pmatrix}
\]
and
\[
	F \begin{pmatrix} u \\ v \\ w \\ t \end{pmatrix} = 0.
\]
Moreover, the entries of $F$ in the upper left $3\times 3$ submatrix have bi-degree $(2,2)$ in $(u,v,w; s,t)$, the bottom right entry has bi-degree $(4,0)$, the remaining entries bi-degree $(3,1)$. 
\end{lemma}

\begin{proof}
Consider the Jacobian matrix 
\[
	J = (f_u,f_v,f_w)
\]
and the Hessian matrix
\[
 	H = \begin{pmatrix}
		f_{uu} & f_{uv} & f_{uw} \\
		f_{uv} & f_{vv} & f_{vw} \\
		f_{uw} & f_{vw} & f_{ww}
	       \end{pmatrix}.
\]
In this situation we have
\[
	J (u,v,w)^t = 4 f
\]
and
\[
	H (u,v,w)^t = 3 J^t.
\]
We now consider the matrix
\[
	F = \begin{pmatrix}
		\frac{t^2}{12} H & \frac{-t}{4} J^t \\
		\frac{-t}{4} J & f
	      \end{pmatrix}.
\]
$F$ is symmetric with
\[
	F_{\{t=0\}} = \begin{pmatrix} 0 & 0 \\ 0  & f \end{pmatrix}.
\]
The above calculations show
\[
	F (u,v,w,t)^t = 0.
\]
\end{proof}

\medskip

\emph{Relation of the above construction of symmetric $\Phi_0$'s to the Prym curves of Artin-Mumford type} 

\begin{lemma}\xlabel{lRelationArtinMumfordM}
Performing an appropriate base change in $3\sO (-3) \oplus \sO(-5) \oplus \sO(-4)$ and symmetrically in $3\sO (-3) \oplus \sO(-1) \oplus \sO(-2)$ any matrix $M$ as in \ref{fM} can be brought into the ``Artin-Mumford" form
 \[
\begin{pmatrix}
a & b & 0 & 0 & 0\\
b & c & 0 & 0 & 0\\
0 & 0 & 1 & 0 & 0\\ 
0 & 0 & 0 & 1 & 0\\ 
0 & 0 & 0 & 0 & 1
\end{pmatrix},
\]
with 
\begin{align*}
a &= -u^2-v^2-w^2\\
b &= -uc_u-vc_v-wc_w\\
c &= -c_u^2-c_v^2-c_w^2 + f.
\end{align*}
Choosing $f, c_u, c_v, c_w$ in $M$ appropriately, we can, up to a projective transformation, obtain any smooth conic $a$ and cubic $b$, quartic $c$ in this way.
\end{lemma}

\begin{proof}
The first part is an explicit computation, the second part is obvious from the formulas for $a,b,c$.
\end{proof}

\begin{proposition}\xlabel{pIsomCoker}
Suppose that $M$ is as in \ref{fM} and $a=b=c=0$ with the notation in Lemma \ref{lRelationArtinMumfordM} has no solution; 
then the cokernels of $\Phi_0$ in diagram \ref{dPhizero} and $M$ in diagram \ref{dMNzero} are isomorphic; in particular, by Lemma \ref{lRelationArtinMumfordM}, we get any good sextic Prym curve of Artin-Mumford type via the construction above.
\end{proposition}

\begin{proof}
A computer algebra computation \cite{BB-M2-16} shows that  any $M$ as in \ref{fM} has rank $5$ generically, rank $4$ on a curve $C$ of degree $6$ (depending of course on the parameters $M$ depends on), and rank $3$ only for $a=b=c=0$.
Similarly, $N_0$ has rank $3$ generically, rank $2$ on the same curve $C$, and rank $1$ only if $a=b=c=0$. This shows that the cokernels of
$M$ and $\Phi_0$ are line bundles $\beta$ and $\beta'$ on the same sextic curve $C$ under the assumption that $a=b=c=0$ has no solution. It remains to check whether these line bundles are isomorphic. For this we compare the divisors of zeros of certain canonically given sections of $\beta(1)$ and $\beta'(1)$.

First consider the diagram:

%\newcommand{stackThree}[3]{#2}
%	\begin{matrix}{c}
%	#1 \\ 
%	\oplus \\
%	#2 \\
%	\oplus \\
%	#3
%	\end{matrix}
%	}

%\stackThree{A}{B}{C}
{\small
\xymatrix{
	&&
	0 \ar[d]
	\\
	&&
	\sO(-1) \ar@{=}[r] \ar[d]&
	\sO(-1)\ar[d]^s 
	\\
	0 \ar[r] &
	3\sO (-3) \oplus \sO(-5) \oplus  \sO(-4)   \ar[r]^-M \ar@{=}[d]& 
	3\sO (-3) \oplus \sO(-1) \oplus \sO(-2)    \ar[r]  \ar[d] &
	\beta \ar[d] \ar[r] &
	0
	\\
	&
	3\sO (-3) \oplus \sO(-5) \oplus \sO(-4)  \ar[r]^-{\overline{M}} & 
	3\sO (-3) \oplus \sO(-2)    \ar[r]  \ar[d] &
	\sO_D \ar[r] \ar[d]&
	0
	\\
	&& 0 & 0
	\\
	}
}
\noindent
where $\overline{M}$ is the matrix obtained by erasing the second row of $M$ and
$D$ is  the locus where $\overline{M}$ drops rank. By the diagram it is clear that $D$ is the divisor associated to the section $s$ in $H^0(\beta(1))$ on $C$.
A direct calculation shows that $D$ is defined by 
\[
	u^2+v^2+w^2=c_u u+c_v v+c_w w = 0 \quad \text{(i.e. $a=b=0$).}
\]
Similarly consider the diagram

\begin{center}
\xymatrix{
	&&
	0 \ar[d]
	\\
	&&
	\sO(-1) \ar@{=}[r] \ar[d]&
	\sO(-1)\ar[d]^{s'} 
	\\
	0 \ar[r] &
	\Omega(-2) \oplus \sO(-5)   \ar[r]^-{\Phi_0} \ar@{=}[d]& 
	\Omega(-1) \oplus \sO(-1)    \ar[r]  \ar[d] &
	\beta' \ar[d] \ar[r] &
	0
	\\
	&
	\Omega(-2) \oplus \sO(-5)    \ar[r]^-{\overline{\Phi_0}} & 
	\Omega(-1)  \ar[r]  \ar[d] &
	\sO_{D'} \ar[r] \ar[d]&
	0
	\\
	&& 0 & 0
	\\
	}
\end{center}
where $\overline{\Phi_0}$ is induced by the matrix $\overline{N_0}$ which is obtained by 
erasing the last row of $N_0$. A direct calculation shows that $D'$ is also defined by 
the equations above. 

This proves $D = D'$ and therefore $\beta(1) = \beta'(1)$
\end{proof}

\begin{proof}[Proof of Theorem \ref{tSexticDegeneration}]
Since (i) is a special case of (ii), which in turn is implied by (iii), it suffices to prove (iii) and (iv). For (iv) choose $M$ as \ref{fM} in such a way that we get a good sextic plane Prym curve of Artin-Mumford type (in particular, it splits as a union of two cubics tangent to a conic). Write down the deformation $\Phi$ of the corresponding $\Phi_0$ as constructed in the proof of Proposition \ref{pDeformationPhi} (using Lemmata \ref{lBo1} and \ref{lBo2}), and verify by Macaulay2 that the general fiber of the family of conic bundles defined by $\Phi$ is smooth \cite{BB-M2-16}. Then (iv) of Theorem \ref{tSexticDegeneration} holds by Proposition \ref{pIsomCoker}. Then also (iii) holds by Lemma \ref{lRelationArtinMumfordM}, Proposition \ref{pDeformationPhi} and Proposition \ref{pIsomCoker}. Note that generically, the conditions necessary for the validity of Proposition \ref{pIsomCoker} (that $a=b=c=0$ has no solutions) will continue to hold, and the general fiber of the resulting family of conic bundles associated to $\Phi$ will still be smooth since we verified this in a particular case above by explicit computation.
\end{proof}

\section{A lemma of Colliot-Th\'{e}l\`{e}ne and Totaro and the general case of Theorem \ref{tMain}} \xlabel{sGeneralTheorem}

Notice that in Sections \ref{sDetDeg} and \ref{sGoodDegPrym} we have proved in two different ways that a very general hypersurface $H_{2,2} \subset \PP^2_{(x:y:z)} \times \PP^2_{(u:v:w)}$ is not stably rational, and, in fact, we have proved a little more by \cite[Thm.  1.14]{CT-P16}: we know that such a $H_{2,2}$ does not have universally trivial Chow group of zero cycles. Recall: a smooth projective variety $X$ over $k=\CC$ ($k$ could be a different field in another set-up, though) has universally trivial Chow zero if for any field $L$ containing $k$ 
\[
\mathrm{CH}_0 (X_L) = \ZZ x_L
\]
where $x$ is a $k$-point of $X$; in other words, for any base change $X_L$ to an overfield $L \supset k$, the Chow group of zero cycles is reduced to $\ZZ$. In fact, as explained in \cite[\S 1.2]{A-CT-P}, $X$ has universally trivial Chow zero if and only if for $L=k(X)$, the diagonal point $\delta_L$ is rationally equivalent over $L$ to some constant point $x_L$ for $x\in X(k)$. So it suffices to check the condition for $L$ the function field of $X$. Having universally trivial Chow zero is also equivalent to having an integral Chow theoretic decomposition of the diagonal in the sense of Voisin \cite{Voi15}; i.e., one can write
\[
\Delta_X = Z_1 + Z_2 \; \mathrm{in} \; \mathrm{CH}^{\dim X} (X\times X) 
\] 
where $Z_2 =X \times \{ x\}$ for $x\in X(k)$ and $Z_1$ is supported on $D\times X$ for some proper closed algebraic subset $D\subsetneq X$.

\smallskip

One can inductively prove the $(2,n)$ case, $n\ge 2$, of Theorem \ref{tMain} starting from the $(2,2)$ case. This possibility as well as the proof was kindly communicated to us by Zhiyu Tian as it arises out of a method used by him, Zhi Jiang and Letao Zhang in forthcoming work and we give this proof here with their permission. We thank them very much for this. The main ingredient is a Lemma due to Colliot-Th\'{e}l\`{e}ne and Totaro \cite[Lemma 2.4]{To16} which says the following:

\begin{lemma}\xlabel{lCT-T}
Let $A$ be a discrete valuation ring with fraction field $K$ and algebraically closed residue field $k$. Let $\mathcal{X}$ be a flat proper scheme over $A$. Let $X$ be the general fiber $\mathcal{X}\times_A K$ and $Y$ the special fiber $\mathcal{X}\times_Ak$. Suppose that $X$ is geometrically integral and there is a proper birational morphism $X' \to X$ with $X'$ smooth over $K$. Suppose that there is an algebraically closed field $F$ containing $K$ such that $\mathrm{CH}_0$ of $X_F'$ is universally trivial. Then, for every extension field $l$ of $k$, every zero-cycle of degree zero in the smooth locus of $Y_l$ is zero in $\mathrm{CH}_0(Y_l)$.
\end{lemma}

This can thus be viewed as in extension of the degeneration method in \cite[Thm.  1.14]{CT-P16} to the case where the central fiber may be reducible: in the form made precise in Lemma \ref{lCT-T}, the triviality of Chow zero is preserved also in this set-up.

Now we can apply this in our set-up as follows: suppose, inductively, that we have already proven that a very general hypersurface $H_{2,n}$ of bidegree $(2,n)$ does not have universally trivial Chow zero, the case $n=2$ being settled. We want to prove the assertion for $n+1$. Now arguing by contradiction, assume a very general $H_{2,n+1}$ had universally trivial Chow zero. Then we could find a family  $\mathcal{X}$ as in Lemma \ref{lCT-T} with $X$ a smooth hypersurface of bidegree $(2, n+1)$ and $Y$ the union $H \cup Z$ where $Z$ is a hypersurface of bidegree $(2,n)$, still general in the sense that it has Chow zero universally nontrivial, and $H$ is of bidegree $(0,1)$, i.e., is of the form $\PP^2_{(x:y:z)}\times h$ for a line $h$ in $\PP^2_{(u:v:w)}$. Now we can use an argument in \cite{To16} (after Lemma 2.4) to get a contradiction: to conclude the proof, by Lemma \ref{lCT-T}, it suffices to find a zero cycle of degree $0$ on
\[
Z_{k(Z)}- (Z\cap H)_{k(Z)}
\]
that is not zero in $\mathrm{CH}_0 ((Z \cup H)_{k(Z)})$. This follows if we can show that
\begin{gather}\label{fChow}
\mathrm{CH}_0 ((Z\cap H)_{k(Z)}) \to \mathrm{CH}_0 (Z_{k(Z)})
\end{gather}
is not surjective (use the Mayer-Vietoris sequence for Chow groups to see this). Now the left hand side of \ref{fChow} is just $\ZZ$ since $Z\cap H$ is a conic bundle over the line $h$ in $\PP^2_{(u:v:w)}$, hence a rational surface, and rational varieties have universally trivial Chow zero. But the right hand side of \ref{fChow} is precisely not equal to $\ZZ$ (there is some nontrivial torsion group) since $Z$ does not have universally trivial Chow zero by the inductive assumption. This concludes the proof.


\begin{thebibliography}{9999999999}
\bibitem[AM72]{AM72}
Artin, M., Mumford, D., \emph{Some elementary examples of unirational varieties which are not rational}, Proc. London Math. Soc. \textbf{25}, 3rd series, (1972), 75--95

\bibitem[A-CT-P]{A-CT-P}
Auel, A., Colliot-Th\'{e}l\`{e}ne, Parimala, R., \emph{Universal unramified cohomology of cubic fourfolds containing a plane}, to appear in ``Brauer Groups and Obstruction Problems: Moduli Spaces and Arithmetic", preprint (2013) \href{http://arxiv.org/abs/1310.6705}{arXiv:1310.6705v2 [math.AG]}

\bibitem[BCF04]{BCF04}
Ballico, E., Casagrande, C., Fontanari, C., \emph{Moduli of Prym curves}, Documenta Mathematica 9 (2004), 265--281

\bibitem[Beau77]{Beau77}
Beauville, A., \emph{Prym varieties and the Schottky problem}, Inv. Math. \textbf{41}, (1977), 149--196

\bibitem[Beau77b]{Beau77b}
Beauville, A., \emph{Vari\'{e}t\'{e}s de Prym et Jacobiennes interm\'{e}diaires}, Ann. scient. \'{E}c. Norm. Sup., 4e s\'{e}rie, tome \textbf{10}, no 3 (1977), 309--391

\bibitem[Beau00]{Beau00}
Beauville, A., \emph{Determinantal hypersurfaces}, Michigan Math. J. \textbf{48}, Issue 1 (2000), 39--64

\bibitem[Beau15]{Beau15}
Beauville, A., \emph{A very general quartic double fourfold or fivefold is not stably rational}, Algebraic Geometry \textbf{2} (4) (2015), 508--513

\bibitem[BB-M2-16]{BB-M2-16}
B\"ohning, Chr. \& H.-Chr. Graf v. Bothmer, \emph{Macaulay 2 files to ``On stable rationality of some conic bundles and moduli spaces of Prym curves"} available at \href{http://www.math.uni-hamburg.de/home/bothmer/m2.html}{http://www.math.uni-hamburg.de/home/bothmer/m2.html}

\bibitem[BCZ04]{BCZ04}
Brown, G., Corti, A., Zucconi, F., \emph{Birational Geometry of 3-fold Mori Fibre Spaces}, Proceedings of the Fano Conference (2004), Torino, Italy, 235--275

\bibitem[CCC07]{CCC07}
Caporaso, L., Casagrande, C., Cornalba, M., \emph{Moduli of roots of line bundles of curves}, Transactions of the AMS \textbf{359}, 8, (2007), 3733--3768 

\bibitem[Cat81]{Cat81}
Catanese, F., \emph{Babbage's conjecture, contact of surfaces, symmetric determinantal varieties and applications}, Invent. Math. \textbf{63}, (1981), 433-- 465

\bibitem[CT-P15]{CT-P15}
Colliot-Th\'{e}l\`{e}ne \& Pirutka, A., \emph{Rev\^{e}tements cycliques qui ne sont pas stablement rationnels}, preprint (2015), \href{http://arxiv.org/abs/1506.00420}{arXiv:1506.00420 [math.AG]}; to appear in Izvestiya RAN, Ser. Math

\bibitem[CT-P16]{CT-P16}
Colliot-Th\'{e}l\`{e}ne \& Pirutka, A., \emph{Hypersurfaces quartiques de dimension 3 : non rationalit\'{e} stable}, Ann. Sci. \'{E}cole Norm. Sup. (2) \textbf{49} (2016), 371--397

\bibitem[Fa12]{Fa12} 
Farkas, G., \emph{Prym varieties and their moduli}, in: Contributions to Algebraic Geometry, Series of Congress Reports, EMS, (2012), DOI: 10.4171/114-1/8, 215--255

\bibitem[HKT15]{HKT15}
Hassett, B., Kresch, A., Tschinkel, Y., \emph{Stable rationality and conic bundles}, to appear in Mathematische Annalen (2015), online first Springer DOI 10.1007/s00208-015-1292-y

\bibitem[HPT16]{HPT16}
Hassett, B., Pirutka, A., Tschinkel, Y., \emph{Stable rationality of quadric surface bundles over surfaces}, preprint (2016), \href{http://arxiv.org/abs/1603.09262}{arXiv:1603.09262 [math.AG]}

\bibitem[HPT16a]{HPT16a}
Hassett, B., Pirutka, A., Tschinkel, Y., \emph{A very general quartic double fourfold is not stably rational}, preprint (2016), \href{http://arxiv.org/abs/1605.03220}{arXiv:1605.03220 [math.AG]}

\bibitem[HT16]{HT16}
Hassett, B., Tschinkel, Y., \emph{On stable rationality of Fano threefolds and del Pezzo fibrations}, preprint (2016), \href{http://arxiv.org/abs/1601.07074}{arXiv:1601.07074 [math.AG]}

\bibitem[Jar98]{Jar98}
Jarvis, T.J., \emph{Torsion-free sheaves and moduli of generalized spin curves}, Compositio Mathematica \textbf{110} (1998), 291--333

\bibitem[Pi16]{Pi16}
Pirutka, A., \emph{Varieties that are not stably rational, zero-cycles and unramified cohomology}, preprint (2016), \href{http://arxiv.org/abs/1603.09261}{arXiv:1603.09261 [math.AG]}

\bibitem[Ta97]{Ta97}
Tate, J., \emph{Finite flat group schemes}, in: Modular forms and Fermat's last theorem, eds. Gary Cornell, Joseph H. Silverman, and Glenn Stevens, New York: Springer-Verlag, (1997), 121--154.

\bibitem[To16]{To16}
Totaro, B., \emph{Hypersurfaces that are not stably rational}, J. Amer. Math. Soc. \textbf{29} (2016), 883--891

\bibitem[Voi15]{Voi15}
Voisin, C., \emph{Unirational threefolds with no universal codimension 2 cycle},  Invent.
Math. \textbf{201} (2015), 207--237
\end{thebibliography}
\end{document}